\documentclass[a4paper,11pt,reqno]{amsart}
\usepackage{graphicx}
\usepackage{amsmath,amssymb,amsthm}
\usepackage{hyperref}
\usepackage{verbatim}
\usepackage{multicol}
\usepackage{tikz}

\theoremstyle{definition}
\newtheorem{thm}{Theorem}[section]
\newtheorem{exm}[thm]{Example}
\newtheorem{defi}[thm]{Definition}

\newtheorem{lemm}[thm]{Lemma}
\newtheorem{prop}[thm]{Proposition}
\newtheorem{cor}[thm]{Corollary}

\numberwithin{equation}{section}

\DeclareMathOperator{\diam}{diam}

\DeclareMathOperator{\curv}{curv}

\DeclareMathOperator{\im}{Im}

\DeclareMathSymbol{\C}{\mathalpha}{AMSb}{"43}
\DeclareMathSymbol{\D}{\mathalpha}{AMSb}{"44}
\DeclareMathSymbol{\R}{\mathalpha}{AMSb}{"52}
\DeclareMathAlphabet{\mathpzc}{OT1}{pzc}{m}{it}

\begin{document}
\title[Conformal Grushin spaces]{Conformal Grushin spaces}
\author{Matthew Romney}
\address{Department of Mathematics, University of Illinois at Urbana-Champaign, 1409 West Green St.,
Urbana, IL 61801}
\email{romney2@illinois.edu}

\date{}
       
\maketitle
\begin{abstract}
We introduce a class of metrics on $\mathbb{R}^n$ generalizing the classical Grushin plane. These are length metrics defined by the line element $ds = d_E(\cdot,Y)^{-\beta}ds_E$ for a closed nonempty subset $Y \subset \mathbb{R}^n$ and $\beta \in [0,1)$. We prove that, assuming a H\"older condition on the metric, these spaces are quasisymmetrically equivalent to $\mathbb{R}^n$ and can be embedded in some larger Euclidean space under a bi-Lipschitz map. Our main tool is an embedding characterization due to Seo, which we strengthen by removing the hypothesis of uniform perfectness. In the two-dimensional case, we give another proof of bi-Lipschitz embeddability based on growth bounds on sectional curvature. 
\end{abstract}

\section{Introduction}\label{sec:intro}

This paper is motivated by the bi-Lipschitz embedding problem: to find useful conditions that guarantee the existence (or non-existence) of a bi-Lipschitz embedding of a given metric space in some Euclidean space of sufficiently high dimension. This is a difficult problem which has been studied by many authors; see for example Heinonen \cite[Sec. 16.4-16.5]{Hein:nonsmooth} and Semmes \cite{Sem:1999} for background. It has drawn attention in part for its connections to theoretical computer science (see Naor \cite{Naor}). In this paper we study a class of metrics on $\mathbb{R}^n$ for which an affirmative answer exists. These are length metrics defined using the line element $ds = d_E(\cdot,Y)^{-\beta}ds_E$ for a given nonempty closed set $Y \subset \mathbb{R}^n$ and $\beta \in [0,1)$ (see Definition \ref{defi:grushin_space}). Here $d_E$ denotes Euclidean distance and $ds_E$ denotes the Euclidean line element. We call these {\it conformal Grushin spaces} because they are defined by a conformal deformation of the Euclidean metric (outside of $Y$) and they generalize the classical Grushin plane. The set $Y$ itself is called the {\it singular set}. 

These spaces can be bi-Lipschitz embedded in some Euclidean space assuming only the additional requirement that the metric satisfy a H\"older condition for the exponent $1-\beta$. As we will see, this is a fairly mild requirement; it is satisfied for instance when $Y$ has empty interior and $\mathbb{R}^n \setminus Y$ consists of finitely many uniform domains. We also study other aspects of the geometry of these spaces; in particular, $n$-dimensional conformal Grushin spaces satisfying the above H\"older condition are quasisymmetrically equivalent to $\mathbb{R}^n$ equipped with the Euclidean metric.

This paper builds in particular on a recent paper by Seo \cite{seo}, in which she gives a criterion for bi-Lipschitz embeddability of a metric space (see Theorem \ref{thm:seo_generalization} below) and applies it to the classical Grushin plane. Our proof is based on the same criterion and shows that her approach is useful for a much broader class of examples. For convenience, we state our main result here; its proof will be given in Section \ref{sec:bilip_embed_Grushin}. 

\begin{thm}\label{thm:grushin_bilip_embed}
Let $n \in \mathbb{N}$, let $Y \subset \mathbb{R}^n$ be nonempty and closed, and let $\beta \in [0,1)$. If the $(Y,\beta)$-Grushin space satisfies the H\"older condition (Definition \ref{defi:holder}), then it is bi-Lipschitz embeddable in some Euclidean space. The bi-Lipschitz constant and target dimension depend only on $\beta$, $n$, and the H\"older constant $H$. 
\end{thm}


In the latter part of our paper, we consider on a more general level the question of what geometric conditions on a space are sufficient to guarantee its bi-Lipschitz embedability in Euclidean space. 
One possible answer can be given in terms of curvature bounds in the sense of Alexandrov geometry. We state it as the following result, which is essentially a more specialized version of Seo's embedding criterion. It applies to complete length metric spaces that satisfy a particular two-sided curvature condition except on a subset $Y$, which as before we may call the {\it singular set}.

\begin{thm}\label{thm:embedding_theorem}
Let $(X,d)$ be a complete length metric space satisfying the following:
\begin{itemize}
\setlength{\itemsep}{1pt}
  \setlength{\parskip}{0pt}
  \setlength{\parsep}{0pt}
\item[(1)] $(X,d)$ is doubling.
\item[(2)] There is a closed subset $Y \neq \emptyset$ of $X$ which admits an $L_1$-bi-Lipschitz embedding into some Euclidean space $\mathbb{R}^{M_1}$ for some $L_1\geq1$ and $M_1 \in \mathbb{N}$.
\item[(3)] The space $X \setminus Y$ is locally an (Alexandrov) space of bounded curvature of (Hausdorff) dimension $n< \infty$ and there exists $A\geq 0$ such that 
\begin{equation}\label{equ:curvature_growth}
-A\,d(x,Y)^{-2} \leq \curv_x X \leq A\,d(x,Y)^{-2}.
\end{equation}
for all $x \in X \setminus Y$.
\end{itemize}
Then $(X,d)$ admits a bi-Lipschitz embedding in some Euclidean space of sufficiently large dimension. The bi-Lipschitz constant and target dimension of this embedding depend only on the constants $L_1$, $M_1$,  $n$, and $A$, and the doubling constant.
\end{thm}
We refer to inequality (\ref{equ:curvature_growth}) as the  {\it (upper and lower) Whitney curvature bounds} because of its relation to a Whitney-type decomposition of $X$ with respect to $Y$ (see Definition \ref{defn:christ_whitney}).
Hypothesis (3) in Theorem \ref{thm:embedding_theorem} means that for all $x \in X \setminus Y$ there exists $r_x>0$ such that the ball $B=B(x,r_x)$ is a geodesic space (equipped with the metric $d|_{B \times B}$) and such that $B$ is a space of bounded curvature $\kappa = A\,d(x,Y)^{-2}$. See Section \ref{sec:curvature_bounds} for further explanation. The proof of Theorem \ref{thm:embedding_theorem} is an application of recent work on quantitative bi-Lipschitz embeddings of spaces of bounded curvature by Eriksson-Bique \cite{Eriksson} and will be given in Section \ref{sec:curvature_bounds}. In that section we will also use Theorem \ref{thm:embedding_theorem} to give a simpler alternative proof of Theorem \ref{thm:grushin_bilip_embed} in the two-dimensional case.



The organization of this paper is as follows. In Section \ref{sec:background} we give an overview of bi-Lipschitz embeddings, and in particular the embeddability criterion of Seo. Our version removes one of the hypotheses of the original theorem; in Section \ref{sec:seo_extension} we show how to modify Seo's proof to obtain our version. In Section \ref{sec:defi_grushin_spaces} we define conformal Grushin spaces and give sufficient conditions for the H\"older condition to be satisfied (Proposition \ref{prop:distance_upper_bound}). In Section \ref{sec:geometric_properties} we study the basic geometry of conformal Grushin spaces, giving estimates of the distance function on different scales. We also prove that these spaces are quasisymmetrically equivalent to $\mathbb{R}^n$, assuming the H\"older condition (Theorem \ref{thm:quasisymmetry}). In Section \ref{sec:bilip_embed_Grushin} we prove Theorem \ref{thm:grushin_bilip_embed}. In Section \ref{sec:curvature_bounds}, we prove Theorem \ref{thm:embedding_theorem} and re-prove Theorem \ref{thm:grushin_bilip_embed} in the two-dimensional case. 

\subsection*{Acknowledgments} The author acknowledges support from National Science Foundation grant DMS 08-38434 “EMSW21-MCTP: Research Experience for Graduate Students” at the University of Illinois.  The author also deeply thanks Jeremy Tyson for suggesting the project, for many hours of stimulating conversation, and for a careful reading of the paper. He also thanks Jang-Mei Wu for helpful discussions, Sylvester Eriksson-Bique for an explanation of his work in \cite{Eriksson}, and the anonymous referee for valuable feedback.

\section{Background} \label{sec:background}

\subsection{Bi-Lipschitz maps}

A map $f: (X,d_X) \rightarrow (Y,d_Y)$ between metric spaces is \textit{bi-Lipschitz} if there exists a constant $L \geq 1$ such that $L^{-1} d_X(x,y) \leq d_Y(f(x), f(y)) \leq L d_X(x,y)$ for all $x,y \in X$. The smallest $L$ for which this is satisfied is called the {\it distortion} of $f$, while we call any such $L$ a {\it bi-Lipschitz constant}.  A bi-Lipschitz mapping is continuous and injective and thus is an embedding in the topological sense.
To avoid excessive verbiage, we make a universal convention that terms such as \textit{embeds}, \textit{embedding}, and \textit{embeddable} when used without qualification in this paper always refer to bi-Lipschitz embeddability into some Euclidean space. 
The metric spaces considered in this paper (excluding Section \ref{sec:seo_extension}) are {\it length metric spaces}; this means that $d(x,y) = \inf \ell(\gamma)$, where $\ell$ denotes the length of a path $\gamma$ and the infimum is taken over the length of all paths $\gamma$ from $x$ to $y$. In  this paper, all paths are assumed to be absolutely continuous with respect to the Euclidean metric.     
 
A basic necessary condition for a space $(X,d)$ to admit a bi-Lipschitz embedding in Euclidean space is that $X$ be \textit{doubling}.  This means there exists a constant $D$ such that for all $x \in X$ and $r>0$, the  ball $B(x,r) := \{y: d(x,y)<r\}$ can be covered by $D$ balls $B(x_j, r/2)$.  It is easy to check that the doubling property is inherited by subsets of a metric space and that the image of a doubling space under a bi-Lipschitz map is doubling.  Any Euclidean space is doubling, so any embeddable space must also be doubling.    However, the doubling condition is not sufficient; the standard counterexample is the Heisenberg group with Carnot-Carath\'eodory metric.  This was first noted by Semmes \cite[Thm. 7.1]{Sem:1996}.  On the other hand, if $(X,d)$ is a doubling space, the \textit{snowflaked} metric space $(X,d^\beta)$, where $\beta \in (0,1)$, can be embedded in some Euclidean space. This is {\it Assouad's embedding theorem} \cite[Thm. 2.6]{Assouad}.

The property of bi-Lipschitz embeddability in Euclidean space is preserved under finite gluings, quantitatively, as shown by the following theorem of Lang and Plaut \cite[Thm. 3.2]{lp:2001}.  

\begin{thm}[Lang--Plaut \cite{lp:2001}] \label{thm:lang_plaut1}
For all $L, L' \geq 1$, $n, n' \in \mathbb{N}$, there exists a constant $\overline{L} \geq 1$ such that the following holds.  Let $(X,d)$ be a metric space, $Z, Z' \subset X$, $f:Z \rightarrow \mathbb{R}^n$ a $L$-bi-Lipschitz embedding, and $f': Z' \rightarrow \mathbb{R}^{n'}$ a $L'$-bi-Lipschitz embedding.  Then there is a $\overline{L}$-bi-Lipschitz embedding $f: Z \cup Z' \rightarrow \mathbb{R}^{n+n'+1}$.  
\end{thm}

Even in the absence of a bi-Lipschitz embedding in Euclidean space, any separable metric space (and thus any doubling space) can be isometrically embedded in $\ell^\infty$.  This is Fr\'echet's embedding theorem \cite[Ch. 12]{hei:lectures}. 

The following theorem of Seo \cite{seo} gives conditions for a metric space to be embeddable.

\begin{thm}[Seo \cite{seo}]\label{thm:seo_generalization}
A metric space $(X,d)$ admits a bi-Lipschitz embedding into some Euclidean space if and only if there exist $M_1, M_2 \in \mathbb{N}$, $L_1, L_2\geq1$ such that the following hold:
\begin{itemize}
\setlength{\itemsep}{1pt}
  \setlength{\parskip}{0pt}
  \setlength{\parsep}{0pt}
\item[(1)]  $(X,d)$ is doubling.
\item[(2)]There is a closed subset $Y \neq \emptyset$ of $X$ which admits an $L_1$-bi-Lipschitz embedding into $\mathbb{R}^{M_1}$.
\item[(3)]  There is a Christ-Whitney decomposition of $X \setminus Y$ such that each cube admits an $L_2$-bi-Lipschitz embedding into $\mathbb{R}^{M_2}$ (see Definition \ref{defn:christ_whitney}).  
\end{itemize}
The bi-Lipschitz constant and target dimension of this embedding depend only on $M_1$, $M_2$, $L_1$, $L_2$, the doubling constant of $X$, and the data of the Christ-Whitney decomposition. 
\end{thm} 
Our statement of Theorem \ref{thm:seo_generalization} is slightly stronger than that given by Seo, which has the additional assumption that the space be uniformly perfect.  This assumption is not used in an essential way in her proof, and Seo asks whether it may be removed \cite[Ques. 5.1]{seo}. This is indeed the case; it will be verified in Section \ref{sec:seo_extension}.  Our proof follows her argument fairly closely, requiring moderate reworking to account for the possibility of a non-uniformly perfect space. We have also simplified the statement of Theorem \ref{thm:seo_generalization} in two minor ways: by removing reference to a doubling measure, and by replacing the technical property of {\it admitting uniform Christ-local embeddings} with the more direct formulation given in hypothesis (3); the equivalence of these is immediate from Theorem \ref{thm:lang_plaut1} above. As a final difference, our definition of a Christ-Whitney decomposition (Definition \ref{defn:christ_whitney}) is slightly more general than that given by Seo; this extra flexibility is needed in proving the main theorems. 

A final definition we will require is that of a quasisymmetric map. A topological embedding $f: (X,d_X) \rightarrow (Y,d_Y)$ is {\it quasisymmetric} if there exists a homeomorphism $\eta: [0,\infty) \rightarrow [0,\infty)$ such that 
$$\frac{d_Y(f(x),f(y))}{d_Y(f(x),f(z))} \leq \eta(t)$$
whenever the distinct points $x,y,z\in X$ satisfy $d_X(x,y) \leq td_X(x,z)$. The function $\eta$ is called the {\it control function}. To check that a given map is quasisymmetric, it suffices to find any function $\eta: [0,\infty) \rightarrow [0,\infty)$ satisfying the above inequality such that $\lim_{t\rightarrow 0} \eta(t)= 0$, not necessarily a homeomorphism. 
The notion of quasisymmetric maps is weaker than that of bi-Lipschitz maps: any $L$-bi-Lipschitz map is quasisymmetric with $\eta(t) = L^2t$. For more background on quasisymmetric embeddings, one may consult the book of Heinonen \cite[Ch. 10-11]{hei:lectures}.


\subsection{Christ-Whitney decompositions} \label{sec:christ_whitney}

The following Christ-Whitney decomposition is a slight generalization of that used by Seo.  It is a metric space version of the classical Whitney decomposition theorem, which states that any open subset $\Omega \subsetneq \mathbb{R}^n$ can be written as the union of dyadic cubes with disjoint interior such that the diameter of each cube is comparable to the distance to the boundary; see Stein \cite[Sec. VI.1]{stein:1970}.  Similar Whitney-type decompositions for metric spaces have been considered by other authors, although these typically require only finite overlap instead of disjointness; see for instance Semmes \cite[Prop. 6.4]{Sem:1999}.  

The Christ-Whitney decomposition is based on the following theorem of Christ \cite[Thm. 11]{christ:1990}, which says that any doubling  
metric space may be decomposed as a system of ``dyadic cubes.''  

\begin{thm}[Christ decomposition \cite{christ:1990}]\label{thm:christ_decomposition}
Let $(X,d)$ be a doubling 
metric space, and let $\delta, c_0, C_1>0$ be constants satisfying $0<\delta +c_0 < 1/4$  and $C_1>(1-\delta)^{-1}$.  There exists a collection of open subsets  $\{Q_\mu^k: k \in \mathbb{Z}, \mu \in I_k\}$ 
such that 
\begin{itemize}
\setlength{\itemsep}{1pt}
  \setlength{\parskip}{0pt}
  \setlength{\parsep}{0pt}
\item[(1)] $\bigcup_{\mu \in I_k} Q_\mu^k$ is dense in $X$ for all $k \in \mathbb{Z}$.
\item[(2)] For any $\mu, \nu \in I_k$, $k ,l \in \mathbb{Z}$ with $l \geq k$, either $Q_\nu^l \subset Q_\mu^k$ or $Q_\nu^l \cap Q_\mu^k = \emptyset$.  
\item[(3)] For each $Q_\mu^k$, there exists $x_\mu^k \in Q_\mu^k$ such that $B(x_\mu^k, c_0\delta^k) \subset Q_\mu^k \subset B(x_\mu^k, C_1 \delta^k)$. 
\end{itemize} 
The sets $Q_\mu^k$ are referred to as \textit{(Christ) cubes}.  
\end{thm}
Christ's original statement was in the context of metric measure spaces and gave stronger conclusions, stating in rough terms that the measure is not concentrated near the boundary of the Christ cubes.  However, the formulation we have given is sufficient for the present paper.  The proof given in \cite{christ:1990} also proves this weaker statement.
An inspection of Christ's proof also yields the constraints for $\delta, c_0$, and $C_1$ we have given. 

By selecting suitable Christ cubes we can form the following \emph{Christ-Whitney decomposition}, based on that used by Seo \cite[Lemma 2.1]{seo}. Our version introduces a distance factor (the constant $a$ in the data) to allow for decompositions on smaller scales. Having this level of control on the size of Christ-Whitney cubes relative to their distance to the boundary will be important in proving both Theorem \ref{thm:grushin_bilip_embed} and Theorem \ref{thm:embedding_theorem}.   

\begin{defi}[Christ-Whitney decomposition]\label{defn:christ_whitney}
Let $(X,d)$ be a doubling  
metric space.  For an open subset $\Omega \subsetneq X$, a \textit{Christ-Whitney decomposition of} $\Omega$ \textit{with data} $(\delta, c_0, C_1, a)$, where $\delta \in (0,1)$, $C_1 > c_0>0$, and $a \geq 4$, is a collection $M_\Omega$ of open subsets of $X$ satisfying the following properties:
\begin{itemize}
\setlength{\itemsep}{1pt}
  \setlength{\parskip}{0pt}
  \setlength{\parsep}{0pt}
\item[(1)] $\bigcup M_\Omega$ is dense in $\Omega$.
\item[(2)] $Q \cap Q' = \emptyset$ for all $Q, Q' \in M_\Omega$, where $Q \neq Q'$.   
\item[(3)] For any $Q \in M_\Omega$, there exists $x \in \Omega$ and $k \in \mathbb{Z}$ such that $$B(x, c_0\delta^k) \subset Q \subset B(x, C_1\delta^k)$$ and
\begin{equation} \label{equ:christ_distance}
\displaystyle (a-2)C_1 \delta^k \leq d(Q,X\setminus\Omega) \leq \left(\frac{aC_1}{\delta}\right) \delta^k.
\end{equation}    
\end{itemize}
The sets $Q \in M_\Omega$ are referred to as \textit{(Christ-Whitney) cubes}.    
\end{defi}
\begin{lemm}\label{lemm:christ_cubes}
Let $X$ be a doubling 
metric space and let $\Omega \subsetneq X$ be open.  For all $\delta, c_0, C_1, a>0$ satisfying $a\geq 4$, $0<\delta +c_0 < 1/4$  and $C_1>(1-\delta)^{-1}$, there exists a Christ-Whitney decomposition of $\Omega$ with data $(\delta, c_0, C_1, a)$. 
\end{lemm}
\begin{proof}
Fix data $(\delta, c_0, C_1, a)$ satisfying the given constraints. Let $\{Q_\mu^k \subset \Omega: k \in \mathbb{Z}, \mu \in I_k\}$ be a Christ decomposition guaranteed by Theorem \ref{thm:christ_decomposition}.  Next, write 
$\Omega = \bigcup_{k=-\infty}^\infty \Omega_k$, where $$\Omega_k = \{x: aC_1\delta^k < d(x, X\setminus\Omega) \leq aC_1 \delta^{k-1}\}.$$  

We make the initial choice of sets $$M_\Omega^0 = \bigcup_{k=-\infty}^\infty \{Q_\mu^k: \mu \in I_k, Q_\mu^k \cap \Omega_k \neq \emptyset\}.$$  Note that if $Q_\mu^k \in M_\Omega^0$, then $d(Q_\mu^k, X\setminus\Omega) \leq aC_1\delta^{k-1}$ and $d(Q_\mu^k, X\setminus\Omega) \geq aC_1\delta^k - \diam Q_\mu^k \geq (a-2)C_1\delta^k$.  Hence any $Q \in M_\Omega^0$ satisfies inequality (\ref{equ:christ_distance}) for this value of $k$.  

Since any two cubes in $M_\Omega^0$ are either disjoint or one contains the other, we may choose the subset $M_\Omega \subset M_\Omega^0$ of cubes which are maximal with respect to set inclusion.  The collection $M_\Omega$ satisfies property (2) as desired.  That $M_\Omega$ satisfies property (1) and the first part of property (3) follows easily from the properties of a Christ decomposition.   
\end{proof}

\section{Definition of conformal Grushin spaces} \label{sec:defi_grushin_spaces}

We begin by reviewing the classical definition. For this portion of the paper we use $(x,y)$ and $(u,v)$ to denote coordinates in $\mathbb{R}^2$ and $(x,y,z)$ to denote coordinates in $\mathbb{R}^3$.
The $\alpha$-Grushin plane $\mathbb{G}_\alpha^2$ is defined as $\mathbb{R}^2$ equipped with the sub-Riemannian (Carnot-Carath\'eodory) metric generated by $X_1 = \partial_x$ and $X_2 = |x|^\alpha\partial_y$.  Explicitly this metric is given by
$$d(z_1,z_2) = \inf \int_{\gamma} \sqrt{dx^2+|x|^{-2\alpha}dy^2},$$
where the infimum is taken over all (absolutely continuous) \textit{}paths $\gamma$ from $z_1$ to $z_2$. The resulting space is a Riemannian manifold except on the set $\{(x,y): x=0\}$, which we call the {\it singular line}. 
The 1-Grushin plane is generally referred to as the (classical) {\it Grushin plane}; this is a basic example of a non-equiregular sub-Riemannian manifold and a useful testing ground in sub-Riemannian geometry. For more background on the classical Grushin plane, see Bella\"iche \cite{Bellaiche}; the case of the more general $\alpha$-Grushin plane has been studied by Franchi and Lanconelli \cite{FL} and Monti and Morbidelli \cite{MonMor1}.

The fact that the Grushin plane $\mathbb{G}^2_1$ admits a bi-Lipschitz embedding in Euclidean space was first proved by Seo \cite{seo}.  An explicit embedding of the Grushin plane with sharp target dimension of 3 was constructed by Wu \cite{Wu}; this result was generalized to $\mathbb{G}_\alpha^2$ for all $\alpha \geq 0$ by the author and Vellis in \cite{RV}, with sharp target dimension of $[\alpha]+2$. Here $[\cdot]$ denotes the floor function. The embedding result of Seo is subsumed by Theorem \ref{thm:grushin_bilip_embed}. 

To motivate the definition of conformal Grushin spaces, consider the map $\varphi: \mathbb{R}^2 \rightarrow \mathbb{R}^2$ defined by
$$(u,v) = \varphi(x, y) = \left( \frac{1}{1+\alpha} |x|^\alpha x, y\right).$$
One can compute the push-forward of the $\alpha$-Grushin line element under $\varphi$ to be
$$ds' = \frac{1}{(1+\alpha)^{\alpha/(1+\alpha)}|u|^{\alpha/(1+\alpha)}}\sqrt{du^2 + dv^2}.$$
This resulting line element is an example of a {\it conformal deformation} of the Euclidean plane outside the singular line. 

For the following definition, recall the notation $d_E$ for Euclidean distance, and $ds_E$ for the Euclidean line element. Notice that $d_E$ is used for distance between both points and sets.

\begin{defi} \label{defi:grushin_space}
Let $n \in \mathbb{N}$, let $Y \subset \mathbb{R}^n$ be a nonempty closed set, and let $\beta \in [0,1)$. The $(Y, \beta)${\it-Grushin space} is the space $\mathbb{R}^n$ equipped with the metric determined by the line element
$$ds = \frac{ds_E}{d_E( \cdot ,Y)^\beta}.$$ 
More explicitly, the distance between two points $z_1, z_2 \in \mathbb{R}^n$ is given by $\inf  \int_\gamma ds,$
the infimum taken over all paths from $z_1$ to $z_2$. If $n=2$ we call this space the {\it $(Y,\beta)$-Grushin plane}.
\end{defi} 


This definition allows the possibility of infinite distance between points; there is no guarantee that the $(Y,\beta)$-Grushin space is actually a metric space. We remark that the same type of metric, though considered only on a proper domain $\Omega \subset \mathbb{R}^2$, has been studied in connection with Lipschitz mappings by Gehring and Martio \cite{GehringMartio85}. 
If we take $\beta = 1$ and restrict to a proper domain $\Omega \subset \mathbb{R}^n$, we arrive at the definition of the {\it quasihyperbolic metric}. However, the behavior of this class of metrics is quite different, since the singular set $Y$ lies at infinite distance from any nonsingular point under the quasihyperbolic metric. As a final remark, other generalizations of the classical Grushin plane to higher dimensions have been studied; see Wu \cite{Wu:grushin} for definitions and additional background.   

In the language of Definition \ref{defi:grushin_space}, the $\alpha$-Grushin plane is isometric, up to rescaling of the metric, to the $(Y, \beta)$-Grushin space with $Y = \{(u,v) \in \mathbb{R}^2: u=0\}$, where as before $\beta = \alpha/(1+\alpha)$. We highlight one more simple example.
\begin{exm}\label{exm:grushin_cone}
Take $Y = \{0\} \subset \mathbb{R}^2$. Representing points in the $(Y,\beta)$-Grushin plane in polar coordinates by $(r,\theta)$, the map
$$(x, y, z) = \left(r^{1-\beta}\cos\theta, r^{1-\beta}\sin \theta, r^{1-\beta} \sqrt{(1-\beta)^{-2}-1}\right)$$
is a path-isometry between the $(Y,\beta)$-Grushin plane and a cone $S$ in $\mathbb{R}^3$ with angular defect (total curvature) $2\pi\beta$. In this case the intrinsic metric on $S$ is bi-Lipschitz equivalent to the Euclidean metric. Composing the above map with the projection map into the $(x,y)$-plane gives a bi-Lipschitz map between $(\mathbb{R}^2,d_Y)$ and $(\mathbb{R}^2, d_E)$.
\end{exm}

Our main results require the following additional assumption on the $(Y,\beta)$-Grushin metric. From here on we will return to using $x,y,z$ to denote points in a metric space rather than coordinates. 
For a given nonempty closed set $Y \subset \mathbb{R}^n$, we use $d_Y$ to denote distance in the $(Y,\beta)$-Grushin metric on $\mathbb{R}^n$; in all cases there is a fixed $\beta$ which is clear from context. 
Also, $\ell_Y(\gamma)$ will denote the $(Y,\beta)$-Grushin length of the path $\gamma$, and  $\ell_E(\gamma)$ will denote the Euclidean length of $\gamma$.

\begin{defi} \label{defi:holder}
The $(Y,\beta)$-Grushin space satisfies the {\it H\"older condition} if there exists $H>0$ such that $d_Y(x,y) \leq Hd_E(x,y)^{1-\beta}$ for all $x,y \in \mathbb{R}^n$. That is, $d_Y$ is $(1-\beta)$-H\"older continuous as a function on $(\mathbb{R}^n, d_E)$. We call $H$ the {\it H\"older constant}.
\end{defi}

It will follow from Lemma \ref{lemm:distance_lower_bound} below that, assuming the H\"older condition, the $(Y,\beta)$-Grushin metric generates the same topology as the Euclidean metric, and hence that $\iota: (\mathbb{R}^n, d_E) \rightarrow (\mathbb{R}^n, d_Y)$ is a homeomorphism. 

We give a basic sufficient condition for the $(Y,\beta)$-Grushin plane to satisfy the H\"older condition.
Recall that a domain $\Omega \subset \mathbb{R}^n$ is {\it uniform} (or, more explicitly, {\it $C$-uniform}) if there exists $C>0$ such that for any two points $x, y\in \Omega$, there exists an arc $\gamma$ from $x$ to $y$ such that $\ell_E(\gamma) \leq Cd_E(x,y)$ and $d(z,\partial \Omega) \geq C^{-1}\min\{\ell_E(\gamma_{x,z}), \ell_E(\gamma_{z,y})\}$ for all $z \in \im \gamma$. Here $\gamma_{z,w}$ refers to the subarc of $\gamma$ from $z$ to $w$.  It is straightforward to show that this same defining property holds for points in $\overline{\Omega}$ (see V\"ais\"al\"a \cite[Thm. 2.11]{Vais3}). 

\begin{prop}\label{prop:distance_upper_bound}
Let $X \subset \mathbb{R}^n$ be a nonempty closed set such that $\Omega = \mathbb{R}^n \setminus X$ is the union of finitely many $C$-uniform domains and $\overline{\Omega} = \mathbb{R}^n$. Then for all $\beta \in [0,1)$ and any nonempty closed subset $Y \subset X$, the $(Y,\beta)$-Grushin space satisfies the H\"older condition. The H\"older constant depends only on $C$, $\beta$, and the number of components of $\Omega$, denoted by $N$. More precisely, 
$$d_Y(x,y) \leq \frac{2^\beta C^{2-\beta}N}{1-\beta} d_E(x,y)^{1-\beta}.$$
\end{prop}
\begin{proof}
Let $\Omega_1, \ldots, \Omega_N$ denote the components of $\mathbb{R}^n\setminus X$. Given points $x,y \in X$, consider the straight-line path $\gamma$ connecting them. There exist $k \leq N$ points $x=x_0, x_1, x_2, \ldots, x_k=y$ in $X \cap \im \gamma$ such that $x_j, x_{j-1}$ lie in the closure of the same domain $\Omega_{k_j}$ in $\mathbb{R}^n \setminus X$ for all $1 \leq j \leq k$.  Let $\gamma_j: [0, \ell_j] \rightarrow \mathbb{R}^n$ be a path from $x_{j-1}$ to $x_j$ in $\overline{\Omega}_{k_j}$ as in the definition of uniform domain, parametrized with respect to Euclidean arc length. Then 
\begin{align*} 
\ell_X(\gamma_j) & = \int_0^{\ell_j/2} \frac{dt}{d(\gamma_j(t),X)^{\beta}} + \int_{\ell_j/2}^{\ell_j} \frac{dt}{d(\gamma_j(t),X)^{\beta}}\\
& \leq 2 \int_0^{\ell_j/2} \frac{C dt}{t^{\beta}} = \frac{2C}{1-\beta} \frac{\ell_j^{1-\beta}}{2^{1-\beta}} \leq \frac{2^\beta C^{2-\beta}d_E(x_{j-1},x_j)^{1-\beta} }{1-\beta} .
\end{align*}
It follows that 
$$d_X(x,y) \leq \frac{2^\beta C^{2-\beta}N}{1-\beta} d_E(x,y)^{1-\beta}.$$
Since $Y \subset X$, we have $d_Y(x,y) \leq d_X(x,y)$, and the result follows.
\end{proof}

\begin{exm}
Consider the set $Y = (-\infty,0]\times \{0\} \subset \mathbb{R}^2$. Its complement is the slit plane, which is not a uniform domain. However, $Y$ is a subset of $X = \mathbb{R} \times \{0\}$, whose complement consists of two uniform domains. Such a space still satisfies the H\"older condition. 
\end{exm}



\section{Geometric properties of Grushin spaces} \label{sec:geometric_properties}

We continue now by proving some elementary geometric properties that hold for all $(Y,\beta)$-Grushin spaces. These lead us to a proof that the $(Y,\beta)$-Grushin space is quasisymmetrically equivalent to $\mathbb{R}^n$ with the Euclidean metric when it satisfies the H\"older condition. The corresponding fact was proved for the $\alpha$-Grushin plane by Meyerson \cite{Meyerson} and a generalized class of Grushin planes by Ackermann \cite{Ack}. 

Throughout this section we take $Y$ to be a nonempty closed set in $\mathbb{R}^n$ for some fixed $n \in \mathbb{N}$, and we take $\beta \in [0,1)$. 

\begin{lemm} \label{lemm:point_metric}
For any $x \in \mathbb{R}^n$, 
\begin{align*} \label{equ:beta_metric}
d_Y(x,Y) = \frac{d_E(x,Y)^{1-\beta}}{1-\beta}.
\end{align*}
\end{lemm}
\begin{proof} 
Let $x \in \mathbb{R}^n$. By integrating over a straight-line path that realizes the Euclidean distance from $x$ to $Y$, it is clear that $d_Y(x,Y) \leq (1-\beta)^{-1}d_E(x,Y)^{1-\beta}$. For the reverse inequality, let $\gamma$ be an arbitrary path from $Y$ to $x$. Then 
$$\ell_Y(\gamma) \geq \int_0^{\ell_E(\gamma)} \frac{dt}{t^\beta} \geq \int_0^{d_E(x,Y)} \frac{dt}{t^\beta} = \frac{d_E(x,Y)^{1-\beta}}{1-\beta}.$$
\end{proof}

\begin{lemm}\label{lemm:distance_lower_bound}
Fix $c\geq 0$. For any $x,y \in \mathbb{R}^n$ satisfying $d_E(x,Y) \leq cd_E(x,y)$, 
$$d_Y(x,y) \geq \frac{1}{1-\beta}((1+c)^{1-\beta} - c^{1-\beta})d_E(x,y)^{1-\beta}.$$
\end{lemm}
\begin{proof}
Consider an arbitrary path $\gamma$ from $x$ to $y$, parametrized by Euclidean arc length.  Then
\begin{align*} 
\ell_Y(\gamma) \geq &  \int_0^{d_E(x,y)} \frac{dt}{(t+d_E(x,Y))^\beta} \\ & = \frac{1}{1-\beta}\left((d_E(x,y)+d_E(x,Y))^{1-\beta} - d_E(x,Y)^{1-\beta}\right).
\end{align*}
Taking the infimum over all paths yields the same inequality with $d_Y(x,y)$ in place of $\ell_Y(\gamma)$. If $c=0$, then $d_E(x,Y)=0$ as well and we are done. 

Otherwise, write the above inequality in the form 
$$d_Y(x,y) \geq \frac{1}{1-\beta}\left(\left(1+\frac{d_E(x,Y)}{d_E(x,y)}\right)^{1-\beta} - \left(\frac{d_E(x,Y)}{d_E(x,y)}\right)^{1-\beta}\right)d_E(x,y)^{1-\beta}.$$
Let $h(t) = \left((1+t)^{1-\beta} - t^{1-\beta}\right)$, defined on $(0,c]$. We compute
$$h'(t) = (1-\beta)((1+t)^{-\beta} - t^{-\beta})<0.$$
This shows that $h(t)$ is decreasing, and in particular that $$h\left(\frac{d_E(x,Y)}{d_E(x,y)}\right) \geq h(c).$$ This yields the desired inequality. 
\end{proof}

\begin{thm} \label{thm:quasisymmetry}
Assume the $(Y,\beta)$-Grushin space satisfies the H\"older condition with constant $H$. The identity map $\iota: (\mathbb{R}^n, d_E) \rightarrow (\mathbb{R}^n, d_Y)$ is quasisymmetric. The control function $\eta$ depends only on $\beta$ and $H$.
\end{thm}
\begin{proof}
We need to show that, for all $t \geq 0$, there exists a value $\eta(t)$ such that, for all distinct points $x,y,z \in \mathbb{R}^n$ satisfying $d_E(x,y) \leq td_E(x,z)$, $d_Y(x,y) \leq \eta(t)d_Y(x,z)$, with $\lim_{t \rightarrow 0} \eta(t) = 0$. Let $t \geq 0$ and suppose that $x,y,z \in \mathbb{R}^n$ are three such points. 
We consider two cases.

First, assume that $d_E(x,Y) \leq 2td_E(x,z)$. By Lemma \ref{lemm:distance_lower_bound} there exists $c(t,\beta)>0$ such that $d_Y(x,z) \geq c(t,\beta) d_E(x,z)^{1-\beta}$. Now by the H\"older condition there exists $H$ such that $d_Y(x,y) \leq H d_E(x,y)^{1-\beta}$. It follows that
$$\frac{d_Y(x,y)}{d_Y(x,z)} \leq Hc(t,\beta)\left( \frac{d_E(x,y)}{d_E(x,z)} \right)^{1-\beta} \leq Hc(t,\beta)t^{1-\beta} .$$

Next, assume that $d_E(x,Y)> 2td_E(x,z)$. First consider the straight-line path $\gamma'$ from $x$ to $y$. Notice that for any point $w \in \im \gamma'$, $d_E(w,Y) \geq d_E(x,Y)/2$. From this we obtain
$$d_Y(x,y) \leq \int_{\gamma'} \frac{ds_E}{d_E(\gamma'(\cdot),Y)^\beta} \leq \int_{\gamma'} \frac{ds_E}{(d_E(x,Y)/2)^\beta} = 2^\beta\frac{d_E(x,y)}{d_E(x,Y)^\beta}. $$
Now let $\gamma$ be any path from $x$ to $z$. Then
\begin{align*} 
\ell_Y(\gamma) & = \int_0^{\ell_E(\gamma)} \frac{dt}{d_E(\cdot, Y)^\beta} \geq \frac{d_E(x,z)}{(d_E(x,Y) + d_E(x,z))^\beta} \\ & \geq \left(1 + \frac{1}{2t}\right)^{-\beta} \frac{d_E(x,z)}{d_E(x,Y)^\beta}.
\end{align*}
Since $\gamma$ is arbitrary the same inequality holds with $d_Y(x,z)$ in place of $\ell_Y(\gamma)$. We obtain
$$\frac{d_Y(x,y)}{d_Y(x,z)} \leq 2^\beta \left(1 + \frac{1}{2t}\right)^\beta  \left( \frac{d_E(x,y)}{d_E(x,z)}\right) \leq 2^\beta \left(1 + \frac{1}{2t}\right)^\beta t.$$
Take $\eta(t)$ to be the larger of the two upper bounds. It is easy to see that $\lim_{t \rightarrow 0} \eta(t) = 0$. 
\end{proof}
Since the doubling property is preserved under quasisymmetric maps, we obtain the following corollary.
\begin{cor}\label{cor:doubling}
If the $(Y,\beta)$-Grushin space satisfies the H\"older condition with constant $H$, then it is doubling. The doubling constant depends only on $\beta$, $n$ and $H$.
\end{cor}

\section{Bi-Lipschitz embedding of Grushin spaces} \label{sec:bilip_embed_Grushin}

In this section we will prove Theorem \ref{thm:grushin_bilip_embed}, namely that any $(Y,\beta)$-Grushin space satisfying the H\"older condition is embeddable in Euclidean space. As mentioned earlier, this generalizes Seo's result that the classical Grushin plane is embeddable. 

\begin{proof}[Proof of Theorem \ref{thm:grushin_bilip_embed}]
We will apply Theorem \ref{thm:seo_generalization}. We have already verified the doubling hypothesis in Corollary \ref{cor:doubling}. It is also easy to establish that $Y$ itself is embeddable; this follows immediately from Assouad's embedding theorem \cite[Prop. 2.6]{Assouad} and the fact that $d_Y|Y$ is comparable to $d_E^{1-\beta}|Y$ by the H\"older condition together with Lemma \ref{lemm:distance_lower_bound}.

What remains is to verify the final hypothesis of Theorem \ref{thm:seo_generalization}, namely the existence of uniform bi-Lipschitz embeddings of the cubes of a Christ-Whitney decomposition. In this proof, we use the notation $B(x,r;d)$ to denote balls relative to the metric $d$.

Fix values of $\delta$, $c_0$, and $C_1$, say $\delta = 1/8$, $c_0=1/9$, and $C_1=2$. Consider a Christ-Whitney decomposition of $\mathbb{R}^n \setminus Y$ (with metric $d_Y$) with data $(\delta, c_0, C_1, a)$, where we take $a$ sufficiently large (to be specified later). We follow the notation of Definition \ref{defn:christ_whitney}. Let $Q$ be a Christ-Whitney cube, and $B = B(x,C_1\delta^k; d_Y)$. Take $M = C_1\delta^k$, $L = (\frac{a}{\delta} + 1)C_1\delta^k$, and $\ell = (1-\beta)^{1/(1-\beta)}L^{1/(1-\beta)}$. 

We claim there exists a $J = J(\beta, a)$ such that $B = B(x,M;d_Y) \subset B(x,J\ell; d_E)$ and $J \rightarrow 0$ as $a \rightarrow \infty$. To verify this, let $y \in B$. Let $\gamma$ be any path from $x$ to $y$; then
\begin{align*} 
\ell_Y(\gamma) \geq \int_0^{d_E(x,y)} \frac{ds}{(d_E(x,Y)+d_E(x,y))^\beta} \geq \frac{d_E(x,y)}{(d_E(x,Y) + d_E(x,y))^\beta}. 
\end{align*}
Since this is valid for any $\gamma$, we obtain 
$$M \geq d_Y(x,y) \geq \frac{d_E(x,y)}{(d_E(x,Y) + d_E(x,y))^\beta} \geq \frac{d_E(x,y)}{(\ell + d_E(x,y))^\beta}.$$

For the last inequality, notice that $d_Y(x,Y) \leq L$, which implies $d_E(x,Y) \leq \ell$ by Lemma \ref{lemm:point_metric}. Now $M = (\frac{a}{\delta}+1)^{-1}L = (\frac{a}{\delta}+1)^{-1}(1-\beta)^{-1}\ell^{1-\beta}$. Set $C_2 = (\frac{a}{\delta}+1)^{-1}(1-\beta)^{-1}$; we obtain 
$$1 \leq C_2\left(\frac{\ell}{d_E(x,y)}\right)^{1-\beta}\left(\frac{\ell}{d_E(x,y)} + 1\right)^\beta \leq C_2 \left( \frac{\ell}{d_E(x,y)} + 1\right).$$
It follows that, for sufficiently large $a$, 
$$0< \frac1{C_2}-1 \leq \frac{\ell}{d_E(x,y)} .$$ 
Taking $J = (1/C_2-1)^{-1}$ gives $d_E(x,y) \leq J\ell$; thus $y \in B(x, J\ell;d_E)$ as desired. Moreover, as $a \rightarrow \infty$ we have $C_2 \rightarrow 0$ and hence $J \rightarrow 0$.

Next we estimate the metric $d_Y|B$. Let $y,z \in B$; notice the straight-line path between them lies in $B(x,J\ell;d_E)$. Notice as well that 
$$d_E(x,Y) \geq (a-2)^{1/(1-\beta)} \left(\frac{a}{\delta}+1\right)^{-1/(1-\beta)} \ell .$$
Let $C_3 = (a-2)^{1/(1-\beta)} (\frac{a}{\delta}+1)^{-1/(1-\beta)}$. Observe that $\lim_{a\rightarrow \infty} C_3 = \delta^{1/(1-\beta)}$; hence by taking $a$ to be sufficiently large we may ensure that $C_3 - J$ is positive. Fix such a value $a = a(\beta)$; notice that $a$ is independent of $Y$ and the dimension $n$. We obtain
$$d_Y(y,z) \leq \frac{d_E(y,z)}{(C_3\ell-J\ell)^\beta} = \frac{1}{(C_3-J)^\beta} \frac{d_E(y,z)}{\ell^\beta}. $$
On the other hand, let $\gamma$ be any path from $y$ to $z$. Then
$$\ell_Y(\gamma) \geq \frac{d_E(y,z)}{(\ell+2J\ell)^\beta} =  \frac{1}{(1+2J)^\beta} \frac{d_E(y,z)}{\ell^\beta}.$$
Combining these facts gives
$$ \frac{1}{(1+2J)^\beta} d_E(y,z) \leq \ell^\beta d_Y(y,z) \leq \frac{1}{(C_3-J)^\beta} d_E(y,z).$$
Define $f_B: (B,d_Y) \rightarrow (\mathbb{R}^n, d_E)$ by $f_B(y) = \ell^{-\beta} y$. It is straightforward to check that $f_B$ is bi-Lipschitz with constant depending only on $\beta$. Thus we obtain uniform embeddings of the cubes of a Christ-Whitney decomposition as required by hypothesis (3) in Theorem \ref{thm:seo_generalization}. 
\end{proof}


Note that the H\"older condition is used in two ways in proving Theorem \ref{thm:grushin_bilip_embed}: first, it implies that $(Y,d_Y|Y)$ is a $(1-\beta)$-snowflake of $(Y,d_E)$ and hence that $Y$ is bi-Lipschitz embeddable; second, it is used in the proof of quasisymmetry (Theorem \ref{thm:quasisymmetry}), which implies that $(\mathbb{R}^n, d_Y)$ is doubling. On the other hand, it is not used in constructing the embeddings of Christ-Whitney cubes, and we leave as an open question to what extent the Ho\"lder condition in Theorem \ref{thm:grushin_bilip_embed} may be relaxed.

While Theorem \ref{thm:grushin_bilip_embed} does not attempt to optimize the target dimension, in some simple cases this can be obtained from known results. We give one such example; for its statement, we use the term {\it $\epsilon$-snowflake line} in $\mathbb{R}^2$ for the image of $(\mathbb{R}, |\cdot|^\epsilon)$ under a bi-Lipschitz map into $\mathbb{R}^2$ (relative to the Euclidean metric). Notice that necessarily $\epsilon \in (1/2,1]$. If $\epsilon=1$, this is the more familiar {\it chord-arc line} condition.

\begin{thm}\label{thm:embed_ex1}
Let $Y \subset \mathbb{R}^2$ be an $\epsilon$-snowflake line for some $\epsilon \in (1/2,1]$, let $\widetilde{\beta} = 1-\epsilon$, and let $\widetilde{\alpha} = (\widetilde{\beta}+ \beta - \widetilde{\beta} \beta)/(1-\widetilde{\beta} - \beta + \widetilde{\beta} \beta)$. Then for all $\beta \in [0,1)$, the $(Y,\beta)$-Grushin plane is bi-Lipschitz embeddable in $\mathbb{R}^{[\widetilde{\alpha}]+2}$. 
\end{thm}
\begin{proof}
This follows from the results in \cite{RV}. Let $\widetilde{\beta} = 1-\epsilon$ and let $\widetilde{Y} = \{(0,v): v \in \mathbb{R}\}$. There exists a global bi-Lipschitz map $\Gamma: (\mathbb{R}^2, d_{\widetilde{Y},\widetilde{\beta}}) \rightarrow (\mathbb{R}^2, d_E)$ such that $\Gamma(\widetilde{Y}) = Y$; this is a restatement of \cite[Cor. 1.3]{RV} in terms of our new notation for Grushin-type surfaces. We claim that $\Gamma$ is also a global bi-Lipschitz map between the $(\widetilde{Y}, \widetilde{\beta} + \beta- \widetilde{\beta} \beta)$-Grushin plane and the $(Y,\beta)$-Grushin plane. To check this, note that $d_E(\Gamma(\cdot), \widetilde{Y}) \simeq d_{\widetilde{Y},\widetilde{\beta}}(\cdot, \widetilde{Y}) \simeq d_E(\cdot, \widetilde{Y})^{1-\widetilde{\beta}}$ by Lemma \ref{lemm:point_metric}. Here we write $a \simeq b$ to denote comparability: $a,b \geq 0$ and there exists $C$ such that $C^{-1}b \leq a \leq Cb$. For any path $\gamma$ in $\mathbb{R}^2$, we have
$$\int_{\gamma} \frac{ds_E}{d_E(\cdot, \widetilde{Y})^{\widetilde{\beta}}} \simeq \int_{\Gamma \circ\gamma} ds_E.$$
It follows from this that
$$\int_\gamma \frac{ds_E}{d_E(\cdot, \widetilde{Y})^{(1-\widetilde{\beta})\beta + \widetilde{\beta}}} \simeq \int_{\Gamma \circ \gamma} \frac{ds_E}{d_E(\cdot, Y)^\beta} $$
for any path $\gamma$ in $\mathbb{R}^2$. This shows that $\Gamma$ is bi-Lipschitz as a mapping from the $(\widetilde{Y}, \widetilde{\beta} + \beta- \widetilde{\beta} \beta)$-Grushin plane to the $(Y,\beta)$-Grushin plane. The $(\widetilde{Y}, \widetilde{\beta} + \beta- \widetilde{\beta} \beta)$-Grushin plane is isometric (up to rescaling) to the usual Grushin plane $\mathbb{G}_{\widetilde{\alpha}}^2$, where $\widetilde{\alpha} = (\widetilde{\beta}+ \beta - \widetilde{\beta} \beta)/(1-\widetilde{\beta} - \beta + \widetilde{\beta} \beta)$.  This space is shown to be embeddable in $\mathbb{R}^{[\widetilde{\alpha}]+2}$ in \cite{RV}.
\end{proof}  


\section{Embedding via curvature bounds} \label{sec:curvature_bounds}

This section is dedicated to an alternative approach to applying Seo's embbedability criterion (Theorem \ref{thm:seo_generalization}) based upon curvature bounds in the sense of Alexandrov geometry. We will give the proof of Theorem \ref{thm:embedding_theorem}, stated in the introduction, and use this result to give a simpler proof of Theorem \ref{thm:grushin_bilip_embed} in the two-dimensional case. Our work is based on the following theorem, which is a local version of the main result in Eriksson-Bique \cite{Eriksson} and can be established by the same methods.

\begin{thm}[Eriksson-Bique \cite{Eriksson}]\label{thm:compact_embed}
Let $n \in \mathbb{N}$, $\kappa > 0$. There exist constants $N = N(n), L = L(n)$, such that, for any Alexandrov space of bounded curvature $(X,d)$ of dimension $n$ and $|\curv X| \leq \kappa$, any ball $B(x,r)$ satisfying $$r \leq \min\{d(x, \overline{X} \setminus X)/2, \kappa^{-1/2}\}$$ admits an $L$-bi-Lipschitz embedding in $\mathbb{R}^N$. Here $\overline{X}$ is the metric completion of $X$.
\end{thm}

It seems most natural to state Theorems \ref{thm:embedding_theorem} and \ref{thm:compact_embed} in the language of Alexandrov geometry. We offer a quick review of the definitions; several sources including books by Alexander, Kapovich, and Petrunin \cite{AKP}, Bridson and Haefliger \cite{bh:npc}, and Burago, Burago, and Ivanov \cite{bbi:metric} contain detailed introductions. For brevity we omit definitions of the terms {\it geodesic triangle}, {\it model space of constant curvature $\kappa$}, {\it comparison triangle}, and {\it comparison points}. 
An {\it (Alexandrov) space of bounded curvature $\kappa$}, where $\kappa \geq 0$, is a locally compact length metric space $(X,d)$ such that all local geodesics (locally length minimizing paths) can be locally extended and such that $-\kappa \leq \curv_x X \leq \kappa$ for all $x \in X$. The statement that $-\kappa \leq \curv_x X \leq \kappa$ means that there exists 
a neighborhood $U_x$ of $x$ with $\diam U_x < \pi/(2\sqrt{\kappa})$ such that $(U_x, d|_{U_x\times U_x})$ is a geodesic space and such that the following triangle comparison inequality holds: Let $\Delta$ be a geodesic triangle, $\overline{\Delta} \subset M_{\kappa}^2$ and $\underline{\Delta} \subset M_{-\kappa}^2$ be comparison triangles in the corresponding model spaces of constant curvature $(M_{\kappa}, d_{\kappa})$ and $(M_{-\kappa}, d_{-\kappa})$. Then for all $y,z \in \Delta$ and comparison points $\overline{y}, \overline{z} \in \overline{\Delta}$ and $\underline{y}, \underline{z} \in \underline{\Delta}$, $d_{-\kappa}(\underline{y}, \underline{z}) \leq d(y,z) \leq d_{\kappa}(\overline{y}, \overline{z})$. 

Note that for any point $x$ in a space of bounded curvature and for all $r>0$ sufficiently small, the ball $B(x,r)$ is convex (see \cite[Prop. II.1.4]{bh:npc}), and so the intrinsic metric on this ball is the same as the original metric $d$. One can show from this observation that any connected open subset of a space of bounded curvature $\kappa$ is still a space of bounded curvature $\kappa$, when equipped with the intrinsic metric on that subset.  

In the statements of Theorem \ref{thm:embedding_theorem} and Theorem \ref{thm:compact_embed}, there is no essential loss of generality by considering only smooth Riemannian manifolds instead of spaces of bounded curvature. This is due to a theorem of Nikolaev that the metric on a space of bounded curvature can be bi-Lipschitz approximated with arbitrarily small distortion by a smooth Riemannian metric \cite{Nikolaev91}. In the smooth Riemannian case, it holds that $\curv_x X \leq \kappa$ or $\curv_x X \geq \kappa$ for all $x \in X$ if and only if all sectional curvatures satisfy the same inequality.

Spaces of bounded curvature provide a natural setting to study bi-Lipschitz embeddings since local embeddability is immediate. Hence one needs only to consider the global embeddability of the space. Indeed, assumptions involving some notion of bounded curvature have figured prominently in several papers on the topic, including papers by Lang and Plaut \cite{lp:2001} and Bonk and Lang \cite{BonkLang}, and the aforementioned paper of Eriksson-Bique \cite{Eriksson}.

We proceed now with the proof of Theorem \ref{thm:embedding_theorem}.
The idea is straightforward: we use Theorem \ref{thm:compact_embed} to obtain uniform embeddings of the cubes in a sufficiently fine Christ-Whitney decomposition, and then apply Theorem \ref{thm:seo_generalization}. 

\begin{proof}[Proof of Theorem \ref{thm:embedding_theorem}]
We use the notation in Definition \ref{defn:christ_whitney}. Fix values of $\delta$, $c_0$, and $C_1$, say $\delta = 1/8$, $c_0=1/9$, and $C_1=2$. Let $a = 3A^{1/2} + 5$, where $A$ comes from the Whitney curvature bounds (\ref{equ:curvature_growth}), and consider a Christ-Whitney decomposition of $\Omega = X \setminus Y$ with data $(\delta, c_0, C_1, a)$. For a given Christ-Whitney cube $Q$, let $r = C_1\delta^k$, $R = 3C_1\delta^k$, and $B = B(x,R)$, where $x$ and $k$ are as in property (3) of Definition \ref{defn:christ_whitney}. The requirement that $a-5 \geq 3A^{1/2}$ implies that
$$d(B,Y) \geq d(x,Y) - 3C_1\delta^k \geq (a-2)C_1\delta^k - 3C_1\delta^k = (a-5)C_1\delta^k. $$
In particular, $B$ (equipped with the intrinsic metric) is an Alexandrov space with $\curv B \leq \kappa$ for 
$$\kappa = \frac{A}{(a-5)^2C_1^2\delta^{2k}} \leq \frac{1}{9C_1^2\delta^{2k}} = \frac{1}{R^2}. $$
Moreover, by the completeness of $X$ we see that $3r \leq d(x, \overline{B}\setminus B)$ (recall that $\overline{B}$ is the metric completion of $B$). 
Therefore the ball $B(x,r)$ satisfies the requirements of Theorem \ref{thm:compact_embed}.  Moreover, for any two points in $B(x,r)$, the geodesic connecting them must stay in $B$. This implies that the intrinsic metric on $B(x,r)$ in $B$ is identical to the original metric. We conclude that $B(x,r)$ admits an $L$-bi-Lipschitz embedding in $\mathbb{R}^N$, with $L$ and $N$ independent of the cube $Q$. By Theorem \ref{thm:seo_generalization}, this suffices to prove the embeddability of $X$.

The quantitative statements on the target dimension and bi-Lipschitz constant are easy to establish.  The constants $L$ and $N$ depend only on the dimension $n$, and the data of the Christ-Whitney decomposition depends only on $A$. 
Uniform bounds on target dimension and distortion now follow by Theorem \ref{thm:seo_generalization}.
\end{proof} 

We remark that the doubling hypothesis cannot be removed altogether from Theorem \ref{thm:embedding_theorem}. One may consider infinitely many Grushin half-planes emanating from a common singular line. This space would satisfy the other hypotheses but not be embeddable. 

Theorem \ref{thm:embedding_theorem} leads to a simpler alternative proof of Theorem \ref{thm:grushin_bilip_embed} in the two-dimensional case. As a warm-up, we illustrate this approach with the $\alpha$-Grushin plane.  
We recall that for a Riemannian manifold $X$, $\curv_x X \leq \kappa$ or $\curv_x X \geq \kappa$ if and only if the corresponding inequality holds for all sectional curvatures at $x$.
In the two-dimensional case sectional curvature coincides with Gaussian curvature, which we denote by $K$, and we can use Brioschi's formula valid for a Riemannian metric $ds^2 = Edx^2 + Gdy^2$:
\begin{equation} \label{equ:brioschi}
K = \frac{-1}{2\sqrt{EG}} \left( \left( \frac{G_x}{\sqrt{EG}}\right)_x + \left(\frac{E_y}{\sqrt{EG}}\right)_y \right).
\end{equation}
Here and in the following proposition we use $(x,y)$ again for coordinates in $\mathbb{R}^2$.
Note that this computation is carried out in \cite[Thm. 8.5]{dhlt:2014} for the case $\alpha=1$.

\begin{prop} \label{thm:grushin_embed}
The $\alpha$-Grushin plane is bi-Lipschitz embeddable in some Euclidean space for all $\alpha \geq 0$.  
\end{prop}
\begin{proof}
Fix $\alpha \geq 0$, and take $Y = \{(x,y): x=0\}$ as the singular set. It is well-known that the Grushin plane is doubling, and that the metric on $Y$ satisfies $$d((0,y_1),(0,y_2))= k(\alpha)|y_1-y_2|^{1/(1+\alpha)}$$   
and hence is embeddable by Assouad \cite[Prop. 4.4]{Assouad}.


The remaining part is a computation of curvature on the Riemannian part of the $\alpha$-Grushin plane.  We have $E=1$ and $G= |x|^{-2\alpha}$ in Brioschi's formula (\ref{equ:brioschi}), giving $K = -\alpha(\alpha+1)|x|^{-2}$. 
We can appeal now to Theorem \ref{thm:embedding_theorem}.
\end{proof}

We proceed now with the full proof of the two-dimensional case of Theorem \ref{thm:grushin_bilip_embed}. It is similar in spirit to the proof of Proposition \ref{thm:grushin_embed}; however, we must account for the fact that the function $d_E(\cdot, Y)$ is not smooth in general. 

\begin{proof}[Proof of Theorem \ref{thm:grushin_bilip_embed} ($n=2$ case)]
This proof is an application of Theorem \ref{thm:embedding_theorem}. However, instead of proving embeddability for the $(Y,\beta)$-Grushin plane directly, we will prove the embeddability of a space $(\mathbb{R}^2, \widetilde{d})$ bi-Lipschitz homeomorphic to it. 
The metric $\widetilde{d}$ is defined exactly the same as the $(Y,\beta)$-Grushin metric, except that we replace $d_E(\cdot,Y)$ with a {\it regularized distance function} (see \cite[Sec. VI.2]{stein:1970}), which we denote by $\Delta(\cdot)$. The function $\Delta(\cdot)$ is comparable to $d_E(\cdot,Y)$, is $C^\infty$ in $\mathbb{R}^2 \setminus Y$, and satisfies
\begin{align*}
\left|\partial_{i_1, \ldots, i_k} \Delta(x)\right| \leq & C(i_1, \ldots, i_k) d_E(x,Y)^{1-k} 
\end{align*}
for all $i_1, \ldots i_k$ and all $x \in \mathbb{R}^2$ and constants $C(i_1, \ldots, i_k)$. The comparability constant and the coefficients $C(i_1, \ldots, i_k)$ are independent of $Y$. It is easy to see that $\widetilde{d}$ is bi-Lipschitz equivalent to the $(Y,\beta)$-Grushin metric. 

To apply Theorem \ref{thm:embedding_theorem}, we only need to verify that $(\mathbb{R}^2, \widetilde{d})$ satisfies the Whitney curvature bounds (\ref{equ:curvature_growth}). In the following, write $\Delta$ for $\Delta(x)$. We use Brioschi's formula (\ref{equ:brioschi}), noting that $E=G= \Delta^{-2\beta}$. This yields
\begin{align*} 
|K(x)| & \lesssim \frac{|\partial_1 \Delta|^2 + |\partial_2 \Delta|^2 +|\Delta|(|\partial_{1,1}\Delta|  + |\partial_{2,2}\Delta|)}{|\Delta|^{2(1-\beta)}} \\
& \lesssim \frac{1}{d_E(x,Y)^{2(1-\beta)}} \lesssim \frac{1}{d_Y(x,Y)^2} \lesssim \frac{1}{\widetilde{d}(x,Y)^2} .
\end{align*}
We have written $a \lesssim b$ here to mean that $a,b \geq0$ and $a \leq Cb$ for some constant $C$. This inequality verifies the Whitney curvature bounds and completes the proof.  Note that the implicit constant in this inequality depends only on $\beta$.  Finally, we remark that all the constants involved (included the doubling constant of the space and the target dimension and bi-Lipschitz constant of the embedding of $Y$) depend only on $\beta$ and $H$; this establishes the quantitative part of the theorem.
\end{proof}

While more elaborate, one may carry out the same sectional curvature computation in dimensions greater than two. This yields a bound on sectional curvature of the form
$$|\text{sec}_x(v,w)| \leq A_1 + A_2d_E(x, Y)^{-2}$$
for all $x \in \mathbb{R}^n \setminus Y$ and linearly independent $v,w \in T_x \mathbb{R}^n$. 
Hence the Whitney curvature bounds are satisfied within a fixed distance to $Y$ but may not be for a sequence of points $\{x_j\}$ with $d_E(x_j, Y)$ unbounded. This appears to be a limitation to the applicability of Theorem \ref{thm:embedding_theorem}.

We conclude this section with a non-embeddability example to illustrate the sharpness of the exponent of $-2$ in the Whitney curvature bounds (\ref{equ:curvature_growth}). This is a variation on the classical Grushin plane. 

\begin{exm}
Replace the original Grushin line element with 
$$ds^2 = dx^2 + e^{2/|x|^\epsilon} dy^2 $$
for some $\epsilon >0$. Using Brioschi's formula (\ref{equ:brioschi}) we compute 
$$K =-\epsilon(\epsilon + x^\epsilon(1+\epsilon))x^{-2(1+\epsilon)}.$$ 
So $|K|$ grows at order $\sim x^{-2(1+\epsilon)}$ as $x \rightarrow 0$. The resulting metric space is not doubling. Indeed, a calculation shows that the strip $[2^{-n-1}, 2^{-n}] \times [0,1]$, $n \geq 0$, contains $[2^{n+1}e^{2^{n \epsilon}}]$ disjoint balls of radius $2^{-n-2}$. (As before, $[\cdot]$ denotes the floor function.) The doubling property requires that the number of disjoint balls of radius $2^{-n-2}$ contained in $[0,1] \times [0,1]$ be bounded by $C^n$ for a fixed constant $C$, which does not hold in this case. A non-doubling space cannot be embedded in Euclidean space.
\end{exm}

\section{Proof of Theorem \ref{thm:seo_generalization}} \label{sec:seo_extension} 

This paper concludes with a proof of the stronger version of Seo's embedding theorem, stated in the introduction as Theorem \ref{thm:seo_generalization}.  This section runs very closely to Seo's original paper.  However, in the interest of making the present paper reasonably self-contained, and because we require a number of modifications, we present in full detail a number of her arguments.  

The original theorem had an additional hypothesis that the space be uniformly perfect.  A metric space $(X,d)$ is \textit{uniformly perfect} if there exists $A \geq 1$ with the property that, for all $0<r<\diam X$ and all $x \in X$, there is a point $y \in X$ such that $A^{-1} r \leq d(x,y) \leq r$.  For example, every connected metric space is uniformly perfect with constant $A=1$.  The hypothesis guarantees that the diameter of a ball is comparable to its radius, something which is not true in general.  Hence it is precisely use of this fact which we must avoid.    

Our starting point is the Christ-Whitney decomposition (Definition \ref{defn:christ_whitney}) found in Section \ref{sec:christ_whitney} of this paper.  Our statement differs from the corresponding one in Seo in inequality (\ref{equ:christ_distance}).  The analogous property there (Lemma 2.1 in \cite{seo}) reads 
$$ \diam Q \leq d(Q,X\setminus\Omega) \leq (4C_1A/\delta c_0)\diam Q,$$
where $A$ is the constant of uniform perfectness.  Clearly, the hypothesis of uniform perfectness is needed to obtain the lower bound on $\diam Q$. We note that this inequality is the only place in Seo's proof in which the uniform perfectness assumption is used. 

If the space in question is not uniformly perfect, the actual diameter of a cube $Q$ could be much smaller that $d(Q,X\setminus\Omega)$.  However, if $\diam Q$ is smaller than $c_0 \delta^k$, then $Q$ is isolated from any other cube in $M_\Omega$.  This suggests the strategy of proving Theorem \ref{thm:seo_generalization} by adding an additional point to each Christ-Whitney cube to increase its diameter if needed, and then leveraging Seo's existing proof to prove embeddability of this enlarged space.  On a heuristic level, this approach seems viable, since a non-uniformly perfect space has ``fewer'' points than a uniformly perfect space and so should be easier to embed. There is a technical complication, however, since we no longer know {\it a priori} that this enlarged space is doubling. 

Following this approach, we will define for each cube $Q \in M_\Omega$ the enlarged cube $\widetilde{Q}$.  Identify $X$ with an isometric copy in $\ell^\infty$ and set 
$$\widetilde{Q} = \left\{ \begin{array}{cc} Q \cup \{q\} & \text{ if } \diam Q < c_0 \delta^k/4 \\ Q & \text{ if } \diam Q \geq c_0 \delta^k/4 \end{array} \right. ,$$
where $q \in \ell^\infty$ is any point satisfying $d(q,x) = c_0 \delta^k/4$.  Here $x$ is the same point as in property (3) in Definition \ref{defn:christ_whitney}.  Note that $\widetilde{Q}$ satisfies $\diam \widetilde{Q} \geq c_0 \delta^k/4$.  Moreover, if $Q \neq \widetilde{Q}$, then $\diam \widetilde{Q} \leq c_0 \delta^k/2$.  

Define now $\widetilde{\Omega} = \bigcup \{\widetilde{Q}: Q \in M_\Omega\}$, $M_{\widetilde{\Omega}} = \{\widetilde{Q} : Q \in M_\Omega\}$, and $\widetilde{X} = X \cup \widetilde{\Omega}$.  Note that we have no \textit{a priori} guarantee that $\widetilde{\Omega}$ and $\widetilde{X}$ are doubling, but we will see that this is not an impediment to our proof.  Now, $M_{\widetilde{\Omega}}$ satisfies all the properties in Definition \ref{defn:christ_whitney} and so we may correctly regard $M_{\widetilde{\Omega}}$ as a Christ-Whitney decomposition of the new space $\widetilde{\Omega} \subset \widetilde{X}$.  

The following definitions and lemmas apply equally well to $M_\Omega$ and $M_{\widetilde{\Omega}}$.  However, for ease of notation we will suppress the tilde over elements of $M_{\widetilde{\Omega}}$ except when needed for clarity.  Fix a constant $0< \epsilon<1$.  For two cubes $Q,R \in M_\Omega$, we consider the relative distance $\Delta(Q,R)$, where 
$$\Delta(Q,R) = \frac{d(Q,R)}{\min\{\diam(Q), \diam(R)\}}.$$
Let $Q^* = \{R \in M_\Omega: \Delta(Q,R) < \epsilon\}$ and $Q^{**} = \bigcup \{ R^*\}_{R \in Q^*}.$  We call $Q^*$ the \textit{Whitney ball} of radius $\epsilon$ about $Q$.  We call $Q^{**}$ the \textit{iterated Whitney ball} of radius $\epsilon$ about $Q$.  As a technical point, note that $Q^*$ and $Q^{**}$ are sets of Christ-Whitney cubes rather than points in $X$; the actual sets of points they contain can be written as $\bigcup Q^*$ and $\bigcup Q^{**}$, respectively.  This definition is symmetric: $R \in Q^*$ if and only if $Q \in R^*$, and $R \in Q^{**}$ if and only if $Q \in R^{**}$.  

We now list some basic facts about $Q^*$ and $Q^{**}$.  First we make precise the observation that $Q$ is isolated from other balls when $\diam Q$ is small.

\begin{lemm}\label{prop:whitney_cube}
If $\diam Q \leq c_0 \delta^k/2$, then $Q^* = Q^{**} = \{Q\}$.  Moreover, $R \in Q^*$ for $Q \in M_\Omega$ if and only if $\widetilde{R} \in \widetilde{Q}^*$ for the corresponding $\widetilde{Q} \in M_{\widetilde{\Omega}}$.
\end{lemm}
\begin{proof}
In the case that $\diam Q \leq c_0 \delta^k/2$, we have $Q \subset B(x, c_0 \delta^k/2)$, where $x \in Q$ comes from property (3) of Definition \ref{defn:christ_whitney}.  Now for any cube $R \neq Q$, $R \cap B(x, c_0 \delta^k) = \emptyset$.  We conclude that $d(Q,R) \geq c_0 \delta^k/2 > \epsilon \diam Q$, whence $\Delta(Q,R) > \epsilon$.  

An implication is that if $R \in Q^*$ or $\widetilde{R} \in \widetilde{Q}^*$ for some $R \neq Q$, then it must be the case that $R = \widetilde{R}$.  The second statement follows.
\end{proof}
The next two lemmas correspond to Proposition 2.2 and Proposition 2.3, respectively, in \cite{seo}. The first of these shows that nearby cubes have comparable diameters.
\begin{lemm}\label{prop:diameter}
Let $Q \in M_\Omega$.  Then for all $R \in Q^*$,
$$\displaystyle \frac{a-2}{2}\left(1 + \frac{aC_1}{2\delta c_0} + \epsilon \right)^{-1}\diam R \leq \diam Q \leq \frac{2}{a-2}\left(1 + \frac{aC_1}{2\delta c_0} + \epsilon \right)\diam R .$$
\end{lemm}
\begin{proof}
If $R \neq Q$ and $R \in Q^*$, then by Lemma \ref{prop:whitney_cube} it must hold that $\diam Q \geq c_0 \delta^k/2$.  This combines with inequality (\ref{equ:christ_distance}) in Definition \ref{defn:christ_whitney} to yield the estimate 
$$d(Q,X\setminus\Omega) \leq \frac{2aC_1}{\delta c_0} \diam Q.$$  We have then for all $R \in Q^*$,
\begin{align*}
\diam R & \leq \frac2{a-2} d(R,X\setminus\Omega) \leq \frac2{a-2} \left(\diam Q + d(Q,X\setminus\Omega) + d(R,Q) \right) \\
& \leq \frac{2}{a-2}\left(1 + \frac{aC_1}{2\delta c_0} + \epsilon \right) \diam Q .
\end{align*}
The other inequality follows by symmetry.   
\end{proof}

\begin{lemm}
There exists $N$ such that $|Q^{**}| \leq N$ for all $Q \in M_\Omega$.  Moreover, any point in $\Omega$ is contained in at most $N$ of the cubes $Q^{**}$.  
\end{lemm} 
\begin{proof}
In the case of $M_\Omega$, the first statement follows immediately from the doubling condition and the comparability of diameters of cubes in $Q^{**}$.  Without the doubling condition known to hold for $\widetilde{X}$, we obtain the same result for $M_{\widetilde{\Omega}}$ by Lemma \ref{prop:whitney_cube}.

For the second statement, we recall that $R \in Q^{**}$ if and only if $Q \in R^{**}$.  Suppose that $x \in R$; then $\{Q: x\in Q^{**}\} = \{Q: Q \in R^{**}\}$.  The result follows.    
\end{proof}

The final observation we make here is that if the hypothesis that the Christ-Whitney cubes $Q$ each admit a bi-Lipschitz embedding with uniform bi-Lipschitz constant and embedding dimension is satisfied, then the same is true for the cubes $\widetilde{Q}$, although with larger bi-Lipschitz constant and target dimension in general.  This follows from Theorem \ref{thm:lang_plaut1}.  

At this point, we refer back to Seo's original paper for the remainder of the proof, although it must now be understood as being applied to the enlarged space $\widetilde{X}$ instead of $X$.  To summarize this briefly, the strategy is to construct Lipschitz cutoff functions on the sets $Q^{**}$ and use a coloring argument to pass from local to global embeddability.  This is very much in the spirit of the proof of Assouad's embedding theorem \cite[Prop. 2.6]{Assouad}.  Seo's original proof carries through in our current setting without complication, although the particular constants need to be altered to match those obtained in Lemma \ref{prop:diameter}.


\bibliographystyle{plain}
\bibliography{biblio}

\def\cprime{$'$}
\begin{thebibliography}{10}

\bibitem{Ack}
Colleen Ackermann.
\newblock An approach to studing quasiconformal mappings on generalized
  {G}rushin planes.
\newblock {\em Ann. Acad. Sci. Fenn.}, 40(1):305--320, 2015.

\bibitem{AKP}
Stephanie Alexander, Vitali Kapovich, and Anton Petrunin.
\newblock {\em Alexandrov geometry}.
\newblock 2014.

\bibitem{Assouad}
Patrice Assouad.
\newblock Plongements lipschitziens dans {${\bf R}\sp{n}$}.
\newblock {\em Bull. Soc. Math. France}, 111(4):429--448, 1983.

\bibitem{Bellaiche}
Andr{\'e} Bella{\"{\i}}che.
\newblock The tangent space in sub-{R}iemannian geometry.
\newblock In {\em Sub-{R}iemannian geometry}, volume 144 of {\em Progr. Math.},
  pages 1--78. Birkh\"auser, Basel, 1996.

\bibitem{BonkLang}
Mario Bonk and Urs Lang.
\newblock Bi-{L}ipschitz parameterization of surfaces.
\newblock {\em Math. Ann.}, 327(1):135--169, 2003.

\bibitem{bh:npc}
Martin Bridson and Andre Haefliger.
\newblock {\em Metric spaces of non-positive curvature}, volume 319 of {\em
  Grundlehren der Mathematischen Wissenschaften}.
\newblock Springer-Verlag, Berlin, 1999.

\bibitem{bbi:metric}
Dmitri Burago, Yuri\ Burago, and Sergei Ivanov.
\newblock {\em A course in metric geometry}, volume~33 of {\em Graduate Studies
  in Mathematics}.
\newblock American Mathematical Society, Providence, RI, 2001.

\bibitem{christ:1990}
Michael Christ.
\newblock A {$T(b)$} theorem with remarks on analytic capacity and the {C}auchy
  integral.
\newblock {\em Colloq. Math.}, 60/61(2):601--628, 1990.

\bibitem{dhlt:2014}
Noel Dejarnette, Piotr Haj{\l}asz, Anton Lukyanenko, and Jeremy Tyson.
\newblock On the lack of density of {L}ipschitz mappings in {S}obolev spaces
  with {H}eisenberg target.
\newblock {\em Conform. Geom. Dyn.}, 18:119--156, 2014.

\bibitem{Eriksson}
Sylvester Eriksson-Bique.
\newblock Quantitative bi-{L}ipschitz embeddings of bounded curvature manifolds
  and orbifolds.
\newblock {\em arXiv preprint arXiv:1507.08211}, 2015.

\bibitem{FL}
Bruno Franchi and Ermanno Lanconelli.
\newblock Une m\'etrique associ\'ee \`a une classe d'op\'erateurs elliptiques
  d\'eg\'en\'er\'es.
\newblock {\em Rend. Sem. Mat. Univ. Politec. Torino}, (Special Issue):105--114
  (1984), 1983.
\newblock Conference on linear partial and pseudodifferential operators
  (Torino, 1982).

\bibitem{GehringMartio85}
F.~W. Gehring and O.~Martio.
\newblock Lipschitz classes and quasiconformal mappings.
\newblock {\em Ann. Acad. Sci. Fenn. Ser. A I Math.}, 10:203--219, 1985.

\bibitem{hei:lectures}
Juha Heinonen.
\newblock {\em Lectures on analysis on metric spaces}.
\newblock Universitext. Springer-Verlag, New York, 2001.

\bibitem{Hein:nonsmooth}
Juha Heinonen.
\newblock Nonsmooth calculus.
\newblock {\em Bull. Amer. Math. Soc. (N.S.)}, 44(2):163--232, 2007.

\bibitem{lp:2001}
Urs Lang and Conrad Plaut.
\newblock Bilipschitz embeddings of metric spaces into space forms.
\newblock {\em Geom. Dedicata}, 87(1-3):285--307, 2001.

\bibitem{Meyerson}
William Meyerson.
\newblock The {G}rushin plane and quasiconformal {J}acobians.
\newblock {\em arXiv preprint arXiv:1112.0078}, 2011.

\bibitem{MonMor1}
Roberto Monti and Daniele Morbidelli.
\newblock Isoperimetric inequality in the {G}rushin plane.
\newblock {\em J. Geom. Anal.}, 14(2):355--368, 2004.

\bibitem{Naor}
Assaf Naor.
\newblock {$L\sb 1$} embeddings of the {H}eisenberg group and fast estimation
  of graph isoperimetry.
\newblock In {\em Proceedings of the {I}nternational {C}ongress of
  {M}athematicians. {V}olume {III}}, pages 1549--1575. Hindustan Book Agency,
  New Delhi, 2010.

\bibitem{Nikolaev91}
Igor Nikolaev.
\newblock Bounded curvature closure of the set of compact {R}iemannian
  manifolds.
\newblock {\em Bull. Amer. Math. Soc.}, 24(1):171--177, 1991.

\bibitem{RV}
Matthew Romney and Vyron Vellis.
\newblock Bi-{L}ipschitz embedding of generalized {G}rushin plane.
\newblock {\em Math. Res. Lett.}, to appear.

\bibitem{Sem:1996}
Stephen Semmes.
\newblock On the nonexistence of bi-{L}ipschitz parameterizations and geometric
  problems about {$A\sb \infty$}-weights.
\newblock {\em Rev. Mat. Iberoamericana}, 12(2):337--410, 1996.

\bibitem{Sem:1999}
Stephen Semmes.
\newblock Bilipschitz embeddings of metric spaces into {E}uclidean spaces.
\newblock {\em Publ. Mat.}, 43(2):571--653, 1999.

\bibitem{seo}
Jeehyeon Seo.
\newblock A characterization of bi-{L}ipschitz embeddable metric spaces in
  terms of local bi-{L}ipschitz embeddability.
\newblock {\em Math. Res. Lett.}, 18(6):1179--1202, 2011.

\bibitem{stein:1970}
Elias Stein.
\newblock {\em Singular integrals and differentiability properties of
  functions}.
\newblock Princeton Mathematical Series, No. 30. Princeton University Press,
  Princeton, N.J., 1970.

\bibitem{Vais3}
Jussi V{\"a}is{\"a}l{\"a}.
\newblock Uniform domains.
\newblock {\em Tohoku Math. J. (2)}, 40(1):101--118, 1988.

\bibitem{Wu}
Jang-Mei Wu.
\newblock Bilipschitz embedding of {G}rushin plane in {${\bf R}\sp 3$}.
\newblock {\em Ann. Scuola Norm.-Sci.}, 14(2):633--644, 2015.

\bibitem{Wu:grushin}
Jang-Mei Wu.
\newblock Geometry of {G}rushin spaces.
\newblock {\em Illinois J. Math.}, 59(1):21--41, 2015.

\end{thebibliography}

\end{document}